\documentclass[12, reqno]{amsart}
\usepackage{amsmath}
\usepackage{amscd,amsthm,amsfonts,amsopn,amssymb,verbatim}
\usepackage{mathrsfs}
\numberwithin{equation}{section}

\newtheorem{theorem}{Theorem} 

\newtheorem{proposition}[theorem]{Proposition}

\newtheorem{definition}{Definition}
\newtheorem{lemma}{Lemma}

\theoremstyle{remark}
\newtheorem{remark}{Remark}

\def\be{\begin{equation}}
\def\ee{\end{equation}}
\def\vp{\varphi}
\def\ve{\varepsilon}

\allowdisplaybreaks
\begin{document}
\setlength{\parskip}{2pt}
\begin{abstract}
The purpose of this Note is to provide a deterministic implementation of the random wave model for the number of nodal domains in the context of the two-dimensional torus.  The approach is based on recent work due to Nazarov and Sodin and arithmetical properties of lattice points on circles.
\end{abstract}

\title[On toral eigenfunctions and the random wave model]
{On toral eigenfunctions and the random wave model}
\author{Jean Bourgain$^*$}
\address{(J. Bourgain) Institute for Advanced Study, Princeton, NJ 08540}
\email{bourgain@math.ias.edu}
\thanks{$^*$  This work was partially supported by NSF grants DMS-0808042 and DMS-0835373}
\maketitle

\section
{Introduction}

This Note originates from the work of Nazarov and Sodin (\cite {N-S} and \cite{S}) on the behavior of the number of nodal domains of
random eigenfunctions at high energy.
It was sown in \cite{N-S} that the number $N_E$ of a random eigenfunction os $S^2$ of eigenvalue $E$ obeys the so-called random wave
model (RWM) for large $E$ and, with large probability, the ratio $4\pi\frac{N_E}E$ is close to a constant $\sigma>0$.
According to the Bogomolny-Schmit \cite{B-S1}, \cite{B-S2} prediction, this number $\sigma$ can be computed based on a bond
percolation model leading to a conjectured value
\be\label{1.1}
\sigma =\frac {3\sqrt 3-5}\pi \approx 0.0624...
\ee
While the \cite{N-S} work establishes in particular the positivity of $\sigma$, it does not shed any light on its actual value.
Note that \eqref{1.1} is considerably smaller than the general (deterministic) upper bound provided by Pleijel's inequality
\be\label{1.2}
\operatornamewithlimits{\lim\sup}\limits_{n\to\infty} \frac Nn \leq\Big(\frac 2j\Big)^2 =0.691..
\ee
with $j$ the smallest positive zero of the Bessel function $J_0$ and $n\asymp \frac E{4\pi}$ the wave number (see also
\cite {B} for a small improvement).

Better upper bounds on $\sigma$ may be obtained by evaluation of certain geometric parameters using Kac-Rice type
arguments.
It was shown in particular in \cite{K} that $\sigma\leq \frac 1{\sqrt 2\pi}= 0.225...$ by computation of the expected
number of horizontal tangencies to the nodal set.
The same bound may be gotten from its expected total curvature (cf. \cite {Ber}).

In what follows, we do not intend to study further the RWM or the Bogomolny-Schmit heuristics.
Rather, we are interested in a deterministic implementation of the RWM in certain situations.
The idea is very simple.
Assume $-\Delta f = Ef$ an eigenfunction for large $E$.
Fixing some base point $x\in M$ in the manifold $E$, we are considering restrictions $f_x$ of $f$ to neighborhoods of $x$
of the order $0\big(\frac 1{\sqrt E}\big)$ (in fact $\frac R{\sqrt E}$ with $R$ slowing growing to infinity with $E$).
In certain instances, one may then show that the ensemble $(f_x)_{x\in M}$ resembles that of a Gaussian random wave
function.
It turns out that for $M=\mathbb T^2$ the 2-dimensional flat torus and eigenfunctions
\be\label{1.3}
 f(x) =\sum_{|\xi|^2=E} a_\xi e(x\cdot\xi) \quad \big(e(a)=e^{2\pi ia}\big)
\ee
$(a_{-\xi}=\bar a_\xi)$ where $\mathcal E_E=\{\xi\in\mathbb Z^2; |\xi|^2=E\}$ satisfies suitable arithmetical assumptions and
$\sum_{|\xi|^2 =E} |a_\xi|^2 \delta_{\xi/\sqrt E}$ becomes well-distributed on the unit circle, this idea may be worked out
rather easily.
On the arithmetic side, we rely on angular equidistribution results \cite{E-H} (see also \cite {F-K-W} and related
references) and also the recent work \cite{B-B} on additive relations in the set $\mathcal E_E$.
Naturally, one runs into stability problems for the number of nodal domains when perturbing slightly the eigenfunctions,
but these analytical issues have been already completely addressed as part of the remarkable work of  Sodin and Nazarov.
In particular, extensive use is made from the results in \cite {S}.

Recall also that from a result due to A.~Stern \cite{St} (see also \cite{L}), there is no nontrivial lower bound on the number of nodal domains
for $E\to\infty$, which may equal two.
Thus for eigenfunctions \eqref{1.3}, some further assumptions are needed.
Possibly, the equidistribution of the measures $\sum_{|\xi|^2=E} |a_\xi|^2 \delta_{\xi/\sqrt E}$ on $S^1$ may suffice, but
we are only able to establish this in certain cases (for instance assuming $E$ has a bounded number of prime factors and
also in a statistical sense, i.e. for `most' $E$).

Beyond the arithmetical input and the results from \cite{S}, our analysis is essentially straightforward.
No effort has been made to obtain  quantitatively more refined results.
A more general outlook on the approach is discussed in the last section.

Let us return to our model $\mathbb T^2$ and be more specific.

Assume $E\in \mathbb Z_+$ a large odd integer which is a sum of 2 squares; we assume moreover $E$ of the form
\be\label{1.4}
E=\prod p_\alpha^{e_\alpha} \qquad (e_\alpha \geq 1)
\ee
where its prime factors $p_\alpha \equiv 1 (\text{mod\,} 4)$.
Denote
\be\label{1.5}
\mathcal E =\mathcal E_E=\{\xi=(\xi_1, \xi_2)\in\mathbb Z^2; \xi_1^2 +\xi_2^2 =E\}.
\ee
Identifying $(\xi_1, \xi_2)$ with the Gaussian integer $\xi_1+i\xi_2\in \mathbb Z+i\mathbb Z$ and denoting $p_\alpha
=\pi_\alpha \bar\pi_\alpha$ the factorization of $p_\alpha$ in Gaussian primes, the set $\mathcal E$ is obtained as
\be\label{1.6}
\Big\{\prod_\alpha \pi_\alpha^j \bar\pi_\alpha^{e_\alpha-j}; 0\leq j\leq e_\alpha\Big\}
\ee
up to multiplication by $\pm 1$ and $\pm i$. In particular
\be\label{1.7}
|\mathcal E| = 4.\prod (1+e_\alpha)=W.
\ee

Writing $E=\lambda^2$ and
\be\label{1.8}
\pi_\alpha =|\pi_\alpha|\, e^{i\theta_\alpha}
\ee
we obtain
\be\label{1.9}
\xi=\lambda e^{i\psi} \text { for } \ \xi\in\mathcal E_E
\ee
whit angles
\be\label{1.10}
\psi =\sum_\alpha (2j_\alpha -e_\alpha)\theta_\alpha \text { and } \ 0\leq j_\alpha \le e_\alpha
\ee
up to multiples of $\frac\pi 2$.

The eigenfunctions of the Laplacian $-\Delta$ on $\mathbb T^2$ with eigenvalue $E$ are obtained as trigonometric
polynomials
\be\label {1.11}
f = \sum_{\xi\in\mathcal E_E} a_\xi e(x.\xi).
\ee
Let us consider for simplicity the eigenfunction
\be\label {1.12}
\sum_{\xi \in \mathcal E_E} e(x.\xi).
\ee
Our considerations in the remainder of the paper carry over verbatim to the situation \eqref{1.11} with $|a_\xi|, \xi\in
\mathcal E$ equal and more general statements with arbitrary coefficients $(a_\xi)_{\xi\in\mathcal E}$ will be discussed later.

Our aim is to show that under suitable assumptions on $E\to\infty$, the number of nodal domains of \eqref{1.12} obeys the
RWM.
These assumptions are of arithmetical nature and may be loosely formulated as follows

\begin{itemize}
\item [(D)] The points $\{\lambda^{-1} \xi; \xi\in\mathcal E_E\}$ become equidistributed on the unit circle for
$E\to\infty$.

\item [(I)] There are not to many non-trivial additive relations among the elements of $\mathcal E$.
\end{itemize}

While we only need (D) without further quantification, a more precise form of (I) will be required (See Definition 
\ref{definition1}
 and Proposition \ref{proposition1}).
Properties (D) and (I) may be addressed by classical results in number theory.
By \eqref{1.10}, (D) relates to angular distribution of Gaussian primes and we will refer to the results from \cite
{E-H}.
A powerful tool to deduce bounds on the number of additive relations is provided by \cite{E-S-S} on unit equations
\be\label{1.13}
a_1\xi_1+\cdots+ a_\ell \, \xi_\ell =1
\ee
with $\xi_1, \ldots, \xi_\ell$ taken in a multiplicative subgroup $G$ of $\mathbb C^*$ of bounded 
rank (though the available results require some further assumption on the number of prime factors of $E$ to be
applicable to our problem).
Alternatively, one may use the `statistical' results on additive relations proven in \cite{B-B} to treat the case of
`typical' $E$.
Precise statements will be given in section 4 (Theorems \ref{theorem2}, \ref{theorem3}, \ref{theorem4}).

\section
{Local Analysis of the eigenfunction}

Let $\mathbb T^2$ be equipped with normalized measure and let
\be\label{2.1} 
f(x) =\frac 1{\sqrt W} \sum_{\xi\in\mathcal E} e(x.\xi)
\ee
with $\mathcal E=\mathcal E_E$, $E=\lambda^2$ and $W=|\mathcal E|$.
We always assume $W\to \infty$ with $E\to \infty$.

In what follows, we will need several parameters, chosen in a particular 
order, that will be viewed as $O(1)$ for fixed
$E$ and eventually will tend to infinity with $E\to \infty$ at sufficiently slow rate.

Let $1\ll K=o(W)$ be a first large parameter and subdivide $\lambda S^1$ in arcs of size $\frac \lambda K$, leading to a
corresponding partition
\be\label{2.2}
\mathcal E =\bigcup^K_{k=1} \mathcal E^{(k)}.
\ee
More specifically, we subdivide the first quadrant of $\lambda S^1$ and partition the other regions by reflection and
symmetry.
According to (D), assume that
\be\label{2.3}
\Big(\frac 1K-\ve_1\Big) W< |\mathcal E^{(k)}|<\Big(\frac 1K+\ve_1\Big)W
\ee
for each $k= 1, \ldots, K$, where $\ve_1= \ve_1(K)$.
Choose a point $\xi^{(k)} \in\mathcal E^{(k)}$, letting $\xi^{(k')} =-\xi^{(k)}$ if $\mathcal E^{(k')}=-\mathcal E^{(k)}$.

Let $R\gg 1$ be another parameter and denote
\be\label{2.4}
\zeta^{(k)} =\frac R\lambda \xi^{(k)}.
\ee
Hence $|\zeta^{(k)}|=R$.
Fixing $x\in\mathbb T^2$, translate $x$ by $\frac R\lambda y$ with $y=(y_1, y_2)\in [-\frac 12, \frac 12]^2$ and write 
\be\label{2.5}
F_x(y)=f\Big(x+\frac R\lambda y\Big)= \frac 1{\sqrt K} \sum^K_{k=1} f_k(x, y) e(\zeta^{(k)} . y)
\ee
with
\be\label{2.6}
f_k(x, y)=\sqrt{\frac KW} \sum_{\xi\in\mathcal E^{(k)}} 
e\Big(\xi.x+\Big(\frac {R\xi}\lambda -\zeta^{(k)}\Big).y\Big).
\ee
Denote further
\be\label{2.7}
\vp(y) =\vp_x (y) =\frac 1{\sqrt K} \sum^K_{k=1} c_k(x) e(\zeta^{(k)}.y)
\ee
with
\be\label{2.8}
c_k(x)=f_k(x, 0).
\ee

Our next goal is to show the following

\begin{itemize}
\item[(i)]  For most $x$, $\vp_x$ is a perturbation of $F_x$ considered as function of $y$,
\item[(ii)] The random vector $\{c_k(x)\}_{1\leq k\leq K}$ with $x$ ranging in $\mathbb T^2$ has approximatively the same
distribution as the Gaussian vector $\{g_k\}_{1\leq k\leq K}$, with $g_1, \ldots, g_k$ IID normalized complex
Gaussians, subject to the reality condition $g_{k'} = \bar g_k $ for $\zeta^{(k')} =-\zeta^{(k)}$.
At this point, we will then be able to rely on the results from \cite {S}.
\end{itemize}

Note that in (ii), we should see $K$ as fixed and the distributional approximation sufficient in order for the relevant Gaussian
estimates from \cite {S} to carry over.

\begin{lemma} \label{lemma1}
For any fixed $s\geq 1$,
\be\label{2.9}
\Vert F_x -\vp_x \Vert _{L^2_xC_y^s} < R^{C_s} K^{-1}.
\ee
\end{lemma}

\begin{proof}

Since from standard Sobolev estimates, we may bound the $C^s$-norm by the $H^{s+2}$-norm, it suffices to estimate
$$
\Big\Vert f\Big(x+\frac R\lambda y\Big) -\vp_x\Big\Vert_{L^2_xH^s_y} \leq 
C\frac {R^s}{\sqrt W} \Big(\sum^K_{k=1} 
 \sum_{\xi\in\mathcal E^{(k)}}  \Big|\frac R\lambda\xi -\zeta^{(k)}|^2\Big)^{\frac 12} < C\frac {R^{s+1}}K
$$
by \eqref{2.5}--\eqref{2.8}.
\end{proof}

It follows from \eqref{2.9} that after fixing $R$, we may ensure, taking $K$ sufficiently large, that
\be\label{2.10}
\Vert F_x-\vp_x\Vert_{C^S} = o(1)
\ee
for most $x\in\mathbb T^2$.

We now turn our attention to the joint distribution of the vector 
$\{ C_k(x); 1\leq k\leq K\}$,
\be\label{2.11}
C_k(x)=\sqrt {\frac KW} \sum_{\xi\in\mathcal E^{(k)}} e(\xi. x)
\ee
when $x$ ranges in $\mathbb T^2$.

Switching notation a bit, it will be convenient to replace $K$ by $2K$ and 
enumerate $\mathcal E^{(1)}, \ldots, \mathcal E^{(K)},
\mathcal E^{(-1)}, \ldots, \mathcal E^{(-K)}$ with $\mathcal E^{(-k)} =-\mathcal E^{(k)}$.
Obviously $c_{-k}=\bar c_k$.

We specify assumption (I) as follows.

\begin{definition}
\label{definition1}
Fix $0<\gamma<\frac 12$ and some $B\in\mathbb Z_+$.
We say that $\mathcal E$ satisfies property $I(\gamma, B)$ if for $2<\ell\leq B$, the number of non-degenerate additive relations of the
form
\be\label{2.12}
\xi_1+\cdots+ \xi_\ell =0
\ee
among elements $\xi_1, \ldots, \xi_\ell \in\mathcal E$ is at most $N^{\gamma\ell}$.
By `non-degenerate', we mean that in \eqref{2.12} there is no proper vanishing sub-sum.
\end{definition}

There are various ways to select energies $E$ for which above independence property holds and this will be addressed in a later section.

\begin{definition}\label{definition2}
Let $\ve>0$ be a small parameter. 
Say that the random vector $(c_1, \ldots, c_k)$ where the $c_j$ are $\mathbb C$-valued functions, 
$\Vert
c_j\Vert_2 \sim 1$, is $\ve$-Gaussian, provided for any (possibly unbounded) intervals $I_1, J_1, \ldots, I_K, J_K\subset \mathbb R$, we
have
\be\label{2.13}
\Big|\text{mes}[c_1\in I_1 \times J_1, \ldots, c_k\in I_K\times J_K] -\frac 1{(2\pi)^K}\int_
{I_1 \times J_1\times\cdots\times  I_K\times J_K}
 e^{-\frac 12 (x_1^2+ y_1^2+\cdots+x_K^2+y_K^2)} dxdy\Big|<\ve.
\ee
\end{definition}

Choosing $\ve$ sufficiently small (in particular depending on $K$), \eqref{2.13} will enable to approximate for `nice' open sets
$\Omega\subset\mathbb C^K$, $\text{mes} [(c_1, \ldots, c_K)\in\Omega]$ by the corresponding Gaussian measure. We prove

\begin{lemma}\label{lemma2}
Given $\ve>0$, there is $B=B(K, \ve)$ such that if $\mathcal E$ satisfies $I(\gamma, B)$, then the vector function $(c_1, \ldots, c_K)$
on $\mathbb T^2$ as defined above, is $\ve$-Gaussian.
\end{lemma}

\begin{proof}
Well-known arguments reduce the problem to evaluating moments
\be\label{2.14}
\int_{\mathbb T^2} c_1^{r_1} \bar c_1^{s_1} \cdots c_K^{r_K} \bar c_K^{s_K}
\ee
with $r_1, s_1, \ldots, r_K, s_K\in\mathbb Z_+\cup\{0\}$ and
$$
r_1+s_1+\cdots+ r_K+s_K < B_1(K, \ve).
$$

Recall that the procedure consists indeed in evaluating the characteristic function
\be\label{2.15}
\int_{\mathbb T^2} e^{[\alpha_1 \text{ Re\,} c_1+\beta_1 \text { Im\,} c_1 +\cdots+\alpha_K \text { Re\,} c_k+\beta_k\text { Im\,} c_K]}
dx
\ee
with $\alpha_1, \beta_1, \ldots, \alpha_K, \beta_K\in\mathbb R$, subject to some bound $B_2(k, \ve)$.
Those arise by suitable truncations of the Fourier transform of intervals. Then, imposing some bound on $|c_1|, \ldots, |c_K|$, Taylor
expansion of the exponentials in \eqref{2.15}  leads to the expressions \eqref{2.14}.

Substituting \eqref{2.11} in \eqref {2.14} gives
$$
\Big( \frac KW\Big)^{\frac 12(r_1+\ldots+r_K+s_1+\cdots+ s_K)}. \ (2.16)
$$
where (2.16) stands for
 the number of relations
\setcounter{equation}{16}
\be\label{2.17}
\xi_{1,1}+\cdots+\xi_{1, r_1} -\xi_{1, 1}' -\cdots-\xi_{1, s_1}'+
\cdots+\xi_{K,1} +\cdots+\xi_{K, r_K} -\xi_{K, 1}' -\cdots -\xi_{K, s_K}'=0
\ee
with $\xi_{1, 1}, \cdots, \xi_{1, r_1}$ and $\xi_{1, 1}' , \cdots, 
\xi_{1, s_1}' \in \mathcal E^{(1)}, \ldots$

Trivial solutions to \eqref{2.17} are those for which the multi-sets (i.e. taking into account multiplicities)
\be\label{2.18}
\{ \xi_{1, 1}, \ldots, \xi_{1, r_1}, \ldots, \xi_{K, 1}, \ldots, \xi_{K, r_K}\}
\ee 
and
$$
\{ \xi_{1}', \ldots, \xi_{1, s_1}', \ldots, \xi_{K, 1}', \ldots, \xi_{K, s_K}'\}
$$
coincide.
Of course, to have a trivial solution requires
\be\label{2.19}
r_1= s_1,\ldots, r_K=s_K.
\ee

Otherwise, we call the solution non-trivial.

Consider first the contribution of trivial solutions, assuming \eqref{2.19}.

Denote $\Omega =\mathbb T^W$ and define $\tilde c_1, \ldots, \tilde c_K$ on $\mathbb T^K$ by
\be\label{2.20}
\tilde c_k(\Psi)=\sqrt{\frac KW} \sum_{\xi \in\mathcal E^{(k)}} e(\psi_\xi)
\text { with } \Psi =(\psi_\xi)_{\xi\in\mathcal E}.
\ee
Recalling \eqref{2.3} and taking into account the central limit theorem, the distribution of $(\tilde c_1, \ldots, \tilde c_K)$ is
approximatively Gaussian.
The trivial solutions to \eqref{2.17} contribute for
\be\label{2.21}
\int_\Omega|\tilde c_1| ^{ 2r_1} \cdots |\tilde c_K|^{2r_K}\cong \int|g_1|^{2r_1} \cdots |g_k|^{2r_K}.
\ee
Consider next the contribution of non-trivial relations, which will be evaluated using our arithmetical assumption.
 Their number is
obviously bounded by the number of nontrivial relations
\be\label {2.22}
\xi_1+\cdots +\xi_\ell =0 \text { with } \ell = r_1 +\cdots+ r_K+s_1+\cdots+ s_K
\ee 
in elements $\xi$ from $\mathcal E$. 
Partitioning (20) in minimal vanishing sub-sums, at least one of these relations will be non-trivial and therefore of length
$\ell'\geq 3$.
Property $I(\gamma, B)$, $B\geq \ell$, clearly implies the following bound
\be\label {2.23}
C(\ell)\sum_{2\nu \leq\ell -3} 
W^\nu \, W^{\gamma(\ell -2\nu)} < C(\ell) \, W^{\frac \ell 2}. W^{-3(\frac 12-\gamma)}.
\ee
Multiplying with $\big(\frac KW\big)^{\frac \ell 2}$, the resulting contribution of the non-trivial relations \eqref{2.17}
in \eqref{2.14} is therefore at most
\be\label{2.24}
B_3(K, \ve). W^{-(\frac 12-\gamma)}
\ee
which can be made arbitrarily small for $W$ large enough.

This proves Lemma \ref{lemma2}.
\end{proof}
\medskip

\section
{The number of nodal domains}

Consider the eigenfunction \eqref{2.1} on $\mathbb T^2$.

From general theory, the total length of the zero set $Z(f)=\{x\in\mathbb T^2; 
f(x) =0\}$ is $O(\lambda)$ while each nodal domain has area at least $O(\lambda^{-2})$.
In particular, it follows that the number of nodal domains of diameter at 
least $\ve_2^{-1}\lambda^{-1}$ is at most $O(\ve_2\lambda^2)$.
Here $\ve_2$ is a small fixed constant.

Choosing $R$ sufficiently large, it clearly follows from the preceding that
\be\label{3.1}
N_f =\frac {\lambda^2}{R^2} \int_{\mathbb T^2} N_f\Big(x, \frac {R} \lambda\Big)dx+O\Big(\ve_2\lambda^2+\frac {\lambda^2}{R\ve_2}\Big)
\ee
where $N_f$ is the number of nodal domains of $f$ and $N_f(x, \rho)$ the number of nodal domains contained in the open box
$x+(]-\frac \rho 2, \frac \rho 2[\times]-\frac\rho 2, \frac\rho2[)$.

Using our notation \eqref{2.5}, te first term on the rhs of \eqref{3.1} equals
\be\label{3.2}
\frac{\lambda^2}{R^2} \int_{\mathbb T^2} N_{F_x} dx
\ee
with $N_F$ the number of components of $Z(F)$ contained in $]-\frac 12, \frac 12[\times]-\frac 12, \frac 12[$.
Note also that
\be\label{3.3}
N_{F_x}=N_f\Big( x, \frac R\lambda\Big) < O(R^2).
\ee

Let $\vp$ be defined by \eqref{2.7}.
According to Lemma \ref{lemma1}
\be\label{3.4}
\int\Vert F_x -\vp_x\Vert_{C^1} \ dx < R^CK^{-1}.
\ee
Hence, fixing another parameter $\ve_3>0$, it follows that
\be\label{3.5}
\Vert F_x-\vp_x\Vert_{C_1} < \ve_3
\ee
except for $x$ in a subset $V\subset \mathbb T^2$ of measure at most $\ve_3^{-1} R^C K^{-1} <\ve _2$, taking 
$K$ sufficiently large.
Since then, by \eqref{3.3}
$$
\frac{\lambda^2}{R^2} \int_V N_{F_x} dx < O(\lambda^2.|V|)< \ve_2\lambda^2
$$
we may replace \eqref{3.2}, up to $O(\ve_2\lambda^2)$, by
\be\label{3.6}
\frac {\lambda^2}{R^2} \int_{\mathbb T^2\backslash V} N_{\vp_x+\psi_x} \, dx
\ee
where $\Vert\psi_x\Vert_{C^1}< \ve_3$.

For $x\in V$, set $\psi_x=0$.
Since the function $\vp_x$ on $\mathbb R^2$ satisfies $-\Delta \vp_x= R^2\vp_x$, it follows
again from the Faber-Krahn inequality that each nodal domain of $\vp_x$ is of area at least $O(\frac 1{R^2})$ and hence $N_{\vp_x}<0(R^2)$.
Thus
$$
\frac {\lambda^2}{R^2} \int_V N_{\vp_x} dx < O(\lambda^2|V|)<\ve_2 \lambda^2
$$
and in \eqref{3.6}, the integral may be extended to $\mathbb T^2$.
Consequently, we obtain
\be\label{3.7}
\frac{\lambda^2}{R^2} \int_{\mathbb T^2} N_{\vp_x+\psi_x} dx.
\ee
The next step consists in invoking Lemma \ref{lemma2}, which asserts that for $W$ sufficiently large, 
the ensemble $(\vp_x)_{x\in\mathbb T^2}$ has approximately the same distribution  as the 
Gaussian random function
\be\label{3.8}
\Phi_\omega =\frac 1{\sqrt K}\sum^K_{k=1} g_k(\omega) e(\zeta^{(k)} .y)
\ee
with $\{g_k\}$ IID normalized complex Gaussians subject to the condition $g_{k'}=\bar g_k$ for $\zeta^{(k')}=-\zeta^{(k)}$.

We claim that by choosing $\ve$ small enough in Lemma \ref{lemma2}, one can replace \eqref{3.7} by
\be\label{3.9}
\frac{\lambda^2}{R^2} \int N_{\Phi_\omega+\Psi_\omega} d\omega
\ee
where $\Psi_\omega$ is some perturbing function, satisfying
\be\label {3.10}
\Vert\Psi_\omega\Vert_{C^1} < 2\ve_{3}
\ee
and
\be\label{3.11}
-\Delta (\Phi_\omega+\Psi_\omega) =R^2 (\Phi_\omega+\Psi_\omega).
\ee
\medskip

\noindent
{\bf Proof of the claim}

Take $M\sim \sqrt{\log K} \sqrt{\log \frac R{\ve_2}}$ and subdivide the $M$-cube centered at $0$ in $\mathbb C^K$ in cubes $Q_\alpha$ of size $\ve_4=\frac
{\ve_3}{R\sqrt K}$.
Their number is $O\big((M\ve_4^{-1})^{2K}\big)$.

For each $\alpha$, denote
$$
A_\alpha =\{x\in\mathbb T^2; \big(c_k(x)\big)_{1\leq k\leq K}\in Q_\alpha\}
$$
and
$$
B_\alpha =\{\omega\in\Omega; \big(g_k(\omega)\big)_{1\leq k\leq K}\in Q_\alpha\}
$$
with $\Omega$ the probability space on which $\Phi_\omega$ is defined.
According to Lemma \ref{lemma2}, we can
ensure that
$$
\big||A_\alpha|-|B_\alpha|\big| <\ve \text { for each $\alpha$}.
$$
Note that $|B_\alpha|>\delta (K, M, \ve_4)$ and hence, for $\ve$ small enough, we may ensure
\be\label{3.12}
\big|| A_\alpha|-|B_\alpha|\big|< \frac {\ve_2}{2R^2}|B_\alpha|.
\ee
This permits to introduce subsets $A_\alpha' \subset A_\alpha, B_\alpha'\subset B_\alpha$, such that
\be\label{3.13}
|A_\alpha'|=|B_\alpha'|> \Big(1-\frac {\ve_2}{2R^2}\Big)|B_\alpha|
\ee
and a measure preserving map
$$
\tau_\alpha:B_\alpha'\to A_\alpha'.
$$
Define on $\cup B_\alpha'$
\be\label{3.14}
\Psi_\omega(y)=\frac 1{\sqrt K} \sum^K_{k=1} [c_k \big(\tau_\alpha(\omega)\big)-g_k(\omega)] e(\zeta^{(k)}.y)+\psi_{\tau_\alpha(\omega)}
\ee
and set
$$
\Psi_\omega =0 \text { if } \omega\not\in \cup B_\alpha'.
$$
With this construction,
\begin{alignat}{1}\label{3.15}
&\int N_{\Phi_\omega+\Psi_\omega} d\omega = \notag\\
&\sum_\alpha \int_{B_\alpha'} N_{\vp_{\tau_\alpha (\omega)}+\psi_{\tau_\alpha (\omega)}} d\omega +\int_{(\cup B_\alpha')^c} N_{\Phi_\omega} d\omega\notag\\
&=\sum_\alpha \int_{A_\alpha'} N_{\vp_x+\psi_x} dx +O\big(R^2|(\cup B_\alpha')^c|\big)\notag\\
&=\int_{\mathbb T^2} N_{\vp_x+\psi_x} dx+ O\big(R^2|(\cup A_\alpha')^c|\big)+O\big(R^2 |(\cup B_\alpha')^c|\big)
\end{alignat}
where we have used again that
$$
N_{\vp_x+ \psi_x}<O(R^2)\text { and } N_{\Phi_\omega} < O(R^2).
$$

Next
$$
\begin{aligned}
|(\cup B_\alpha')^c|&=\sum_\alpha |B_\alpha\backslash B_\alpha'|\\
&+|\{\omega; \max_{1\leq k\leq K}; |g_k(\omega)|>M\}|\\
&\overset{(3.13)}< \frac {\ve_2}{R^2}\sum |B_\alpha|+ \frac {\ve_2}{R^2}< 2\frac {\ve_2}{R^2}
\end{aligned}
$$
and
$$
|(\cup A_\alpha')^c| =\sum_\alpha |A_\alpha\backslash A_\alpha'|+ |\{x\in \mathbb T^2; \max_{1\leq k\leq K} |c_k(x)|>M\}|.
$$
Again by \eqref{3.13}
$$
\sum_\alpha |A_\alpha\backslash A_\alpha'|\leq \Big(1+\frac{\ve_2}{2R^2}\Big) \sum|B_\alpha|-\sum|B_\alpha'|<\frac{\ve_2}{R^2}.
$$

From Lemma \ref{lemma2}
$$
\begin{aligned}
|\{x\in \mathbb T^2; \max_{1\leq k\leq K} |c_k (x)|>M\}|&\leq \sum^K_{k=1} \text{ mes}[|c_k|>M]\\
&<K(\ve+\text { mes}[|g_k|>M])< \frac {\ve_2}{R^2}.
\end{aligned}
$$
Substituting in \eqref{3.15} gives
\be\label{3.16}
\int N_{\Phi_\omega +\Psi_\omega} d\omega =\int_{\mathbb T^2} N_{\vp_x+\psi_x} dx+O(\ve_2).
\ee
Finally, note that on $B_\alpha'$, by \eqref{3.14} and choice of $\ve_4$
$$
\Vert\Psi_\omega\Vert_{C^1} \leq R\sqrt K\ve_4 +\Vert\psi_{\tau_\alpha(\omega)}\Vert_{C^1}< 2\ve_3.
$$
Also, since either $\Psi_\omega=0$ or $\Phi_\omega+\Psi_\omega=\vp_x+\psi_x$ 
for some $x$, it follows that $-\Delta(\Phi_\omega+\Psi_\omega)=
R^2(\Phi_\omega+\Psi_\omega)$.
This proves the claim. \hfill $\square$

\medskip

At this stage, we are reduced to study the expected number of nodal domains in $]-\frac 12, \frac 12[\times]-\frac 12, \frac 12[$ of the
perturbed Gaussian vector $\Phi_\omega$.

We make use of the work of Nazarov--Sodin and more specifically, several results from \cite {S}.

First there is the stability issue.
Considering the random Gaussian function $\Phi_\omega$ given by \eqref{3.8}, clearly
\be\label{3.17}
\mathbb E_\omega[\Vert\Phi_\omega\Vert_{C^2}] < O(R^2).
\ee
Invoking Lemma 5 from \cite {S}, which is based on the independence of $\Phi_\omega$ and $\nabla\Phi_\omega$, we get some $\beta=\beta(R, \ve_2)>0$
such that
\be\label{3.18}
\min_{x\in[-\frac 12, \frac 12]^2} \max (|\Phi_\omega(x)|, |\nabla\Phi_\omega(x)|)>\beta
\ee
for all $\omega$ outside a set of measure less than $\frac {\ve_2}{R^2}$, hence contributing to
$$
\mathbb E[N_{\Phi_\omega+\Psi_\omega}]
$$
for at most $O(\ve_2)$.

Property \eqref{3.18} is crucial to derive a stability property for the number of nodal domains under perturbation (see \cite {S}, Lemma 6).
Recall that the perturbation $\Psi_\omega$ satisfies $\Vert\Psi_\omega\Vert_{C^1}<\ve_3$.
Taking
\be\label{3.19}
\ve_3 =\beta(R, \ve_2)\frac 1{10R}. 
\ee
Lemma 7 from \cite {S} applied with $f=\Phi_\omega, g=\Psi_\omega$ and $\alpha =2\ve_3$ implies in particular the following
\begin{alignat}{1}\label{3.20}
N_{\Phi_\omega +\Psi_\omega} \geq & \text { number of components of $Z(\Phi_\omega)$ contained in the square}\notag\\
& Q=\Big[-\frac 12, \frac 12\Big]^2 \text { and at distance at least 
$\frac 1{2R}> \frac {2\alpha}\beta$ from $\partial Q$}.
\end{alignat}
Note that since $\Vert\Psi_\omega\Vert_{C^1}<\ve_3$, \eqref{3.18} also implies that
\be\label{3.21}
\min_{x\in Q} \max\big(|(\Phi_\omega+\Psi_\omega)(x)|, |\nabla(\Phi_\omega +\Psi_\omega)(x)|\big)>\frac \beta 2.
\ee
Another application of \cite {S}, Lemma 7 taking $f=\Phi_\omega+\Psi_\omega, g=-\Psi_\omega$ yields conversely that
\be\label{3.22}
\text {Number of components of $Z(\Phi_\omega)$ contained in $\Big[-\frac 12-\frac 1{2R}, \frac 12+\frac 1{2R}\Big]^2\geq N_{\Phi_\omega+\Psi_\omega}$}.
\ee
It follows from the two-sided inequalities \eqref{3.20}, \eqref{3.22} that
\begin{alignat}{1}\label {3.23}
&|\mathbb E[N_{\Phi_\omega}] -\mathbb E[N_{\Phi_\omega+\psi_\omega}|<\notag\\
&\mathbb E\Big[\# \text { components $C$ of $Z(\phi_\omega)$ contained in $\Big[-\frac 12- \frac 1{2R}, \frac 12+ \frac 1{2R}\Big]^2$ for which}\notag\\ 
&\text{dist} (C, \partial Q)<\frac 1R\Big]\\
&+O(\ve_2).\notag
\end{alignat}

Recalling \eqref{3.8} and $\zeta_k =\frac R\lambda \xi_k$ with $\xi_k \in \mathbb Z^2$,
$|\xi_k|=\lambda$, we have
$$
N_{\Phi_\omega} = N\Big(0, \frac R\lambda, f_\omega\Big)
$$
defining
\be\label{3.24}
f_\omega (x) =\frac 1{\sqrt K}\sum^K_{k=1} g_k(\omega) e(\xi_k. x)
\ee
and $N\big(0, \frac R\lambda; f_\omega\big)$ the number of components of $Z(f_\omega)$ contained in \hfill\break
$Q_R =]-\frac R{2\lambda}, \frac R{2\lambda}[\times]-\frac
R{2\lambda}, \frac R{2\lambda}[$.
Thus $f_\omega$ is a Gaussian random eigenfunction of $\mathbb T^2$ of eigenvalue $E$.

The first term in \eqref{3.23} accounts for the number of components $C$ of $f_\omega$ contained in $Q_{R+1}$ and such that 
dist$(C, \partial Q_R)<\frac 1\lambda$.
Each of these components has area at least $O(\frac 1{\lambda^2})$ and it follows from the Kac-Rice formula that
\be\label{3.25}
\mathbb E_\omega[\text{length $(Z_{f_\omega} \cap Q_{R+1})]< O\Big(\frac {R^2}\lambda\Big)$}.
\ee
From these facts, one easily derives that
\be\label{3.26}
\mathbb E\Big[\#\text { components $C$ of $Z(f_\omega)$ s.t. $C\subset Q_{R+1}, \text{dist}(C, \partial Q_R)<\frac 1\lambda\Big]<O(R^{\frac 32}\log R)$}
\ee
(we first exclude those components $C$ of size at least $\log R.\frac{\sqrt R}{\lambda}$, assuming \hfill\break
length $(Z_{f_\omega}\cap Q_{R+1})<\log R.
\frac{R^2}\lambda$ and then exploit the area lower bound of each component).

Hence we proved that
\be\label{3.27}
\mathbb E[N_{\Phi_\omega+\Psi_\omega}]=\mathbb E\Big[N\Big(0, \frac R\lambda, f_\omega\Big)\Big]+O(\ve_2+R^{3/2}\log R).
\ee
The expectation of $N(0, \frac R\lambda, f_\omega)$ in the $\lim_{R\to\infty} \lim_{\lambda\to\infty}$ is given by Theorem 5 in \cite{S} and we get in our
situation
\be\label{3.28-1}
R^2. \nu(\rho) 
\ee
where $\nu(\rho)$ is the constant given by \cite {S}, Theorem 1 associated to the measure $\rho$, which is the limiting spectral measure of our sequence
\eqref{3.24}, in the sense of \cite{S}.
Thus one considers the spectral measure $\rho_\lambda$ associated to \eqref{3.24} defined by
\be\label{3.28}
\widehat{\rho_\lambda} (u-v)=\mathbb E\Big[f_\omega\Big(\frac u\lambda\Big). f_\omega \Big(\frac v\lambda\Big)\Big]
=\frac 1K \sum^K_{k=1} e\Big(\frac{\xi_k}\lambda .(u-v)\Big).
\ee
Hence
$$
\rho_\lambda=\frac 1K\sum^K_{k=1} \rho_{\xi_k\lambda^{-1}}
$$
where $\delta_z$ stands for the Dirac measure at $z\in\mathbb R^2$, $|z|=1$.

Since, by assumption (D) and the construction in Section 2, the measures $\rho_\lambda$ become equidistributed for $\lambda, K\to\infty$, the limiting measure
$\rho$ is he normalized Lebesque measure on the unit circle and $\bar\nu$ is the constant associated to the RWM discussed in the Introduction; i.e.
$\sigma=4\pi\bar\nu$.
Recall \eqref{3.16}, \eqref{3.27} and take say $\ve_2 =R^{-\frac 1{10}}$.
From the preceding, we obtain
$$
\int_{\mathbb T^2} N_{\vp_x+\psi_x} dx =\big(\bar\nu +o(1)\big) R^2
$$
and
\be\label {3.29}
\int_{\mathbb T^2} N_{F_x} dx= \big(\bar\nu+o(1)\big)R^2.
\ee
In view of \eqref{3.1}, \eqref{3.2}, we obtain finally from the choice of $\ve_2$, that
$$
N_f =\big(\bar \nu+o(1)\big) \lambda^2.
$$
Recapitulating the preceding, the order in which the various parameters are chosen is
$$
R, \ve_2, \beta, \ve_3, K, M, \ve_4, \ve, \ve_1, B(K, \ve).
$$
We proved the following

\begin{proposition}\label{proposition1}

Assume $E$ taken in a sequence such that (D) holds for $E\to \infty$ and also, for some fixed $\gamma<\frac 12$, 
condition $I\big(\gamma, B(E)\big)$ with $B(E)\overset {E\to\infty}\longrightarrow \infty$.
Let
$$
f_E=\sum_{\xi\in\mathcal E_E} e(x.\xi)
$$
or, more generally
$$
f_E=\sum_{\xi \in\mathcal E_E} a_\xi e(x.\xi) \text { with } a_{-\xi}=\bar a_\xi, |a_\xi|=1.
$$
Then the number $N_E$ of nodal domains of $f_E$ satisfies
$$
4\pi \frac{N_E}E\to\sigma
$$
\end{proposition}
\medskip

\section
{Arithmetic considerations}

We return to the assumptions (D) of equidistribution and (I) of independence.
Recall that we assumed $E$ of the form
\be\label{4.1}
E=\prod p_\alpha^{e_\alpha}
\ee
with $p_\alpha$ odd, $p_\alpha\equiv 1(\text{mod\,} 4)$.

Let $\pi _\alpha =|\pi_\alpha| e^{i\theta_\alpha}$ and write $\xi =\lambda e^{i\psi_\xi}$ for $\xi\in\mathcal E=\mathcal E_E$,
according to \eqref{1.6}, \eqref{1.10}.

We start with a statistical discussion, considering a `typical' integer $E$ of the above form.

A quantitative form of the required angular equidistribution is established in \cite {E-H} (see Theorem 1).

\begin{lemma}\label{lemma3}

Given $\ve>0$, for almost all integers $E$ considered above, one has a discrepancy bound
\be\label{4.2}
\Delta(E) < |\mathcal E_E|(\log E)^{-\kappa+\ve} \text { where } \kappa =\frac 12\log\frac\pi 2.
\ee
\end{lemma}

Here $\Delta(E)$ is defined as
\be\label{4.3}
\max_{0\leq\alpha<\beta< 2\pi}
\Big|\frac{\beta-\alpha}{2\pi} |\mathcal E_E|-[\#\{\xi\in\mathcal E; 
\psi_\xi \in [\alpha, \beta](\text{mod\,} 2\pi)]\Big|.
\ee
The proof of this result depends on  Kubilius' evaluation of the number of Gaussian primes in a sector and bounds on
multiplicative functions.

Let us also recall that, on average, an integer $E$ that is sum of 2 squares has $\asymp \frac 12\log\log E$ prime factors, 
implying that  $|\mathcal E_E|\sim \sqrt{\log E}$.

Next, we discuss the independence condition, again from  a statistical perspective.
The following statement follows from \cite{B-B}, Theorem 17 and Remark 15.

\begin{lemma}\label{lemma4}
Given $\ell>2$, for most integers $E$ of type \eqref  {4.1}, the number of non-degenerate relations $\xi_1+\cdots+\xi_\ell =0$ among
elements of $\mathcal E_E$ is at most $O(|\mathcal E_E|)$ for $E\to\infty$.
More precisely, given any function $\vp$, $\frac{\vp(u)}u \overset {u\to\infty}\rightarrow \infty$, for most $E$, the
number of non-degenerate solutions is bounded by $\vp(|\mathcal E_E|)$.
\end{lemma}

Obviously, this implies property $I(\gamma, B)$ for any given $\gamma>\frac 13$, for typical $E$ taken large enough.

Note that \cite{B-B}, Theorem 17 follows from the following statement, which in some sense is stronger.

Denote
\be\label{4.4}
\Omega_{X, K} =\big\{E=\prod p_\alpha <X; p_\alpha\equiv 1 (\text{mod\,} 4), p_\alpha>K\big\}.
\ee 
This set satisfies
\be\label{4.5}
|\Omega_{X, K}|\sim \frac X{\sqrt{\log X}\sqrt{\log K}}.
\ee
Theorem 14 in \cite {B-B} asserts then that for fixed $\ell$,
\be\label{4.6}
\lim_{K\to\infty}\lim_{X\to\infty} \frac 1{|\Omega_{X, K}|}|\{E\in\Omega_{X, K}; \text { $\mathcal E_E$
admits a nondegenerate relation of 
length } \ell\}|=0.
\ee

Lemmas \ref{lemma3} and \ref{lemma4} are clearly addressing the assumptions from our theorem in the statistical sense.
Thus we can state

\begin{theorem}\label{theorem2}
The conclusion from Proposition \ref{proposition1} holds for almost all $E\to\infty$ of the form \eqref{1.4}.
\end{theorem}

Our next goal is a deterministic implementation.
We start with the independence assumption.
In certain cases, the desired information is provided by the deep work of 
Evertse-Schlickewei-Schmidt on additive relations in
multiplicative subgroups of $\mathbb C^*$ of bounded rank \cite {E-S-S}, which in turn depends on the subspace theorem.
The result of \cite {E-S-S} states that any unit equation
\be\label{4.7}
a_1g_1+\cdots+ a_\ell g_\ell =1
\ee
with $g_1, \ldots, g_\ell$ taken in a multiplicative
group $G$ over $\mathbb C$ of $\mathbb Q$-rank $r$, has at most exp $\big(c(\ell)(r+1)\big)$
non-degenerate solutions.  Here $c(\ell)$ may be taken
\be\label{4.8}
c(\ell)=(4\ell)^{3\ell}.
\ee
An immediate consequence is the following

\begin{lemma}\label{lemma5}
Let $E=\prod^r_{\alpha=1} p_\alpha^{e_\alpha}$ be as above.
Then the number of non-degenerate relations \eqref{2.12} among elements from $\mathcal E_E$ is 
bounded  by
\be\label{4.9}
\exp \big(c(\ell-1)(2r+1)\big) |\mathcal E_E|.
\ee
\end{lemma}

While estimate \eqref {4.9} does not suffice in general to conclude a condition $I(\gamma, B)$,
it does suffice provided $r=o(\log|\mathcal E_E|)$, i.e., recalling \eqref{1.7}
\be\label{4.10}
\frac 1r \sum^r_{\alpha =1} \log e_\alpha \to \infty.
\ee
This is in particular the case if we fix the prime factors $p_1, \ldots, p_r$ of $E$
and take their exponents $e_\alpha$ large enough.

\begin{remark}\label{remark1}
It has been suggested that the true upper bound for the number of non-degenerate solutions
of \eqref{4.7} may be subexponential in the rank, possibly bounded by
\be\label{4.11}
\exp\big(c(\ell) r^{\beta(\ell)}\big) \text { for some } \beta(\ell)<1.
\ee
If this were true, $I(\gamma, B)$ would hold with any $\gamma>\frac 13$, for all sufficiently 
large $E$.
\end{remark}

\begin{remark}\label {remark2}
If we fix $p_1, \ldots, p_r=1(\text{\rm mod\,} 4)$ and let $E=p_1^{e_1}\cdots p_r^{e_r}$ with $\max e_\alpha\to\infty$, condition (I) is certainly satisfied.
Condition (D) amounts by \eqref{1.10} to equidistribution of the angular set
\be\label{4.12}
\Big\{ 2\sum^r_{\alpha=1} j_\alpha \theta_\alpha; 0\leq j_\alpha\leq e_\alpha\Big\}
\quad (\text {mod\,} 1).
\ee
Hence $\theta_\alpha \not\in 2\pi \mathbb Q$.
This is the case, since otherwise $\cos 2b\theta_\alpha=1$ for some $b\in\mathbb Z_+$, implying that $\cos 2\theta_\alpha$ is an algebraic integer.
But since $tg \theta_\alpha =\frac {\xi_2}{\xi_1} \in\mathbb Q$, $\cos 2\theta_\alpha \in\mathbb Q$ and therefore
$\cos  2\theta_\alpha \in\mathbb Z$, $\theta_\alpha \in\frac\pi 4\mathbb Z$ (contradiction).
\end{remark}

Hence 

\begin{theorem}\label{theorem3}
Assume given $p_1, \ldots, p_r\equiv 1(\text{\rm mod\,} 4)$.
Then the conclusion from Proposition \ref{proposition1} holds when $E$ ranges in the set
$$
\{p_1^{e_1} \ldots p_r^{e_r}, e_1, \ldots, e_r\in\mathbb Z_+\}.
$$
\end{theorem}

\begin{remark}\label{remark3}
We may also state the following property, which results from \eqref{4.7}, \eqref{4.8} and an easy adaptation of the 
proof of Lemma\ref{lemma2}.
\end{remark}

\begin{theorem}\label{theorem4}

Fix $r$ and let $E$ range in a sequence of integers of the form
$$
E=\prod^r_{\alpha=1} p_\alpha^{e_\alpha} \qquad 
\big(p_\alpha\equiv 1 (\text{\rm mod\,} 4)\big).
$$
For each $E$, let 
\be\label {4.13}
f_E=\sum_{\xi\in\mathcal E_E}  a_\xi e(x.\xi)\qquad (a_{-\xi}=\bar a_\xi)
\ee
$$
\sum |a_\xi|^2=1
$$
be arbitrarily chosen, subject to the assumption that the probability measures
\be\label{4.14}
\rho_E =\sum_{\xi\in\mathcal E_E} |a_\xi|^2 \delta_{\lambda^{-1}\xi}\qquad (\lambda^2= E)
\ee
on the unit circle, converge weak$^*$ to the normalized Lebesque measure on $S^1$ for $E\to\infty$.

Denoting $N_E$ the number of nodal domains of $E$, we have that
\be\label{4.15}
\frac{N_E}E\to \bar\nu \text { for } E\to\infty.
\ee
\end{theorem}
\medskip

\section
{Further comments}

An alternative approach would consist in considering `jets' of eigenfunctions.
Thus given an eigenfunction $f_E$ of $\mathbb T^2$, introduce at each point $x\in \mathbb T^2$ the scaled function
\be\label{5.1}
\vp(y) =\vp_x(y) =f\Big(x+\frac y\lambda\Big).
\ee
The function $\vp$ may be $\ve$-approximated on $[|y|<R]$ by truncation of its Taylor expansion at order
$B=B(R, \ve)$, leading to a jet
\be\label{5.2}
J_x=\{D^\alpha \vp_x|_{y=0}\}_{|\alpha|<B}.
\ee
Consider $J_x$ as a random vector in $x\in\mathbb T^2$.
Under assumptions (D) and (I), one may then show that the distribution of $(J_x)_{x\in\mathbb T^2}$ is approximatively the 
same as for the Gaussian random function with circular spectral measure and derive from this that $\frac{N_E}E\to\bar\nu$.
This approach has the advantage of at least conceptually generalizing to real analytic compact manifolds $M$.
Following \cite{S},  Section 2, one considers a map (assuming $\dim M=2$)
$$
\Phi_x =\exp_x \circ I_x:\mathbb R^2 \to M, \Phi_x (0)=x
$$
with $\exp_x: T_xM\to X$ the exponential map and $I_x :\mathbb R^2\to T_x(M)$ a linear Euclidean isometry.
The function $\vp_x$ is then defined by
\be\label{5.3}
\vp_x(y)=f\big(\Phi_x(\lambda^{-1} y)\big).
\ee
But we preferred to follow the procedure adopted earlier because it is more explicit and, in any case, we do not
have examples at this point, other than the flat torus, where the RWM may be implemented deterministically.

This discussion may however be of interest in the (arithmetic) hyperbolic case. (See \cite {G-R-S} for some 
remarkable new results on nodal domains in this setting).

Basically, the required behavior of the $(J_x)_{x\in M}$ may here in some sense be seen as a far generalization of the
Gaussian distribution conjecture of the eigenfunctions.

\end{document}